\documentclass[11pt]{amsart}
\usepackage{amsmath}
\usepackage{amssymb}
\usepackage{amsthm}
\usepackage{mathrsfs}
\usepackage[normalem]{ulem}
\usepackage[all]{xy}
\usepackage{multicol}
\usepackage{lipsum}
\usepackage{enumitem}

\raggedright
\newenvironment{theorem}[2][Theorem]{\begin{trivlist}
\item[\hskip \labelsep {\bfseries #1}\hskip \labelsep {\bfseries #2}]}{\end{trivlist}}

\newenvironment{example}[2][Example]{\begin{trivlist}
\item[\hskip \labelsep {\bfseries #1}\hskip \labelsep {\bfseries #2}]}{\end{trivlist}}
\title[Invariant Subspaces as Kernels]{Constructing Invariant Subspaces\\ as Kernels of Commuting Matrices}
\author{Carl C. Cowen}
\address{Department of Mathematical Sciences\\
Indiana University-Purdue University Indianapolis, Indianapolis, IN 46202, USA}

\email{ccowen@iupui.edu}%
\author{William Johnston}%
\address{Department of Mathematics, Statistics, and Actuarial Science\\
 Butler University, Indianapolis, IN 46208, USA}
\email{bwjohnst@butler.edu}%
\author{Rebecca G. Wahl}%
\address{Department of Mathematics, Statistics, and Actuarial Science\\
 Butler University, Indianapolis, IN 46208, USA}%
\email{rwahl@butler.edu}%
\date{June 7, 2019 }
\thanks{The first author was funded by Simons Foundation Collaboration Grant 358080.\ The authors thank the reviewer for helpful suggestions that improved the exposition.\\ {This is the 2nd author's manuscript of the article published in final edited form as Cowen, C. C., Johnston, W., $\&$ Wahl, R. G. (2019). Constructing invariant subspaces as kernels of commuting matrices. {\it Linear Algebra and Its Applications}, 583, 46--62. https://doi.org/10.1016/j.laa.2019.08.014}}
\newtheorem{thm}{Theorem}
 
\newtheorem{lemma}[thm]{Lemma}
 \newtheorem{cor}[thm]{Corollary}
 \newcommand{\dfn}{\noindent {\bf Definition.\ } }
   \newcommand{\by}{\times}

   \newcommand{\Co}{\ensuremath{\mathbb{C}}} 

    \newcommand{\Cn}{\ensuremath{\mathbb{C}^n}}
   
   \renewcommand{\L}{\ensuremath{\mathcal{L}}}
   
   \newcommand{\M}{\ensuremath{\mathcal{M}}}
   \newcommand{\N}{\ensuremath{\mathcal{N}}}
   
    \newfont{\caps}{cmcsc10}  
  \newfont{\jour}{cmti10}  

\begin{document} \vspace*{-2ex}
  \maketitle
\vspace{-6ex}
\begin{abstract} \noindent {Given an $n \times n$ matrix $A$ over $\mathbb{C}$ and an invariant subspace \N, a straightforward formula constructs an $n \times n$ matrix $N$ that commutes with $A$ and has $\N=\ker N$.  For $Q$ a matrix putting $A$ into Jordan canonical form, $J=Q^{-1}AQ$, we get $\N=Q^{-1}\M$ where $\M=$ ker$(M)$ is an invariant subspace for $J$ with $M$ commuting with $J$. In the formula $M=PZT^{-1}P^t$, the matrices $Z$ and $T$ are $m \times m$ and $P$ is an $n\by m$ row selection matrix.  If \N\ is a marked subspace, 
$m=n$ and $Z$ is an $n\by n$ block diagonal matrix,  and if \N\ is not a marked subspace, then $m>n$  and $Z$ is an  $m\by m$ near-diagonal block matrix.
  \ Strikingly, each block of $Z$ is a monomial of a finite-dimensional backward shift.\ Each possible form of $Z$ is easily arranged in a lattice structure isomorphic to and thereby displaying the complete invariant subspace lattice 
  $\L(A)$ for $A$.}
\end{abstract} 

\noindent It is easy to show, for any $n \times n$  matrix $A$, that $\ker N$ is an $A$-invariant subspace when $N$ commutes with $A$.\ In 1971 Paul Halmos~\cite{Halmos} proved the impressive converse:\ {\it every} invariant subspace of a given $n \times n$ matrix $A$ over $\mathbb{C}$  can be expressed as the kernel of a matrix $N$ in the commutant, $ \{A\}^\prime$,  of $A$.\ An  elegant proof by Ignat Domanov \cite{Domanov} in 2010 followed an earlier simplification by Abdelkhalek Faouzi \cite{Faouzi}.  Related questions on invariant subspaces over arbitrary fields are in \cite{Jones}, \cite{Fripertinger}, and \cite{Added}.\ This paper builds on Halmos' result for matrices over $\mathbb{C}$ and proves, via construction, Theorems A and B stated below. The construction focuses on whether or not the invariant subspace \N\ is ``marked'' -- when there is a Jordan basis for $A$ acting on \N\ that can be extended (by adjoining new vectors) to form a Jordan basis for $A$ on the entire space $\mathbb{C}^n$. The authors of \cite{Rodman}\  show that  \emph{Every invariant subspace for $A$ is marked if and only if for every eigenvalue $\lambda$, the difference between the largest and smallest multiplicity of $\lambda$ as an eigenvalue is no more than $1$}.  The construction for marked subspaces will be fairly easy.\ When \N\ is not marked, the matrix will be strategically ``expanded'' to act on a larger dimensional vector space. The result from \cite{Rodman} will then guarantee that \N, thought of now as living inside the larger vector space, is marked. The marked construction is then employed on the larger space and subsequently projected back down to the original vector space to form the $n \times n$ matrix $N$. \\[1ex]

\begin{theorem}{\bf A.} {\it For a given $n \times n$ matrix $A$ over $\mathbb{C}$ and an $A$-invariant subspace \N, there exists an $n \times n$ matrix $N=QMQ^{-1}$ over $\mathbb{C}$, where $Q$ puts $A$ into Jordan form $J$ and: 
\begin{enumerate}  [label=\alph*)\ ]
  \item $\N=\ker N$;
\item $N \in \{A\}^\prime$;
\item the corresponding $J$-invariant subspace $\M=Q^{-1}\N$ has $\M=\ker M$;
\item $M \in \{J\}^\prime$; 
\item $M=PZT^{-1}P^t$ where $Z$ and $T$ are $m\by m$ matrices for some $m\geq n$, $P$ is an $n\by m$ matrix and  $P^t$ denotes the transpose of $P$;\newpage 
\item $T$ and $P$ provide a standard change of variables and row selection matrix, respectively;
\item $Z$ is a diagonal or near-diagonal block matrix whose non-zero blocks are each a power of a finite-dimensional backward shift.\ 
\end{enumerate}
Furthermore, the subspace \N \ is marked if and only if the construction produces $Z$ and $T$ that are  $n \times n$ with $Z$ block-diagonal and $P$ the identity. } 
\end{theorem}

\begin{theorem}{\bf B.} {\it For a given $n \times n$ matrix $A$ over $\mathbb{C}$, there is a one-to-one correspondence between elements in the lattice of invariant subspaces and elements in a lattice of the matrices $Z$ produced by Theorem A for marked subspaces. This correspondence provides a new ability to construct systematically the full invariant subspace lattice $\L(A)$, as well as the sublattice of hyperinvariant subspaces. }
\end{theorem}

First put $A$ into Jordan canonical form:\ write $A=QJQ^{-1}$, where $J$ has the block form \small $J= \left[ \begin{smallmatrix} J_1 &  &  \vspace{-5pt}\\   & \ddots &   \\ & & J_u \end{smallmatrix} \right]$\normalsize.\ Here, blocks off the main diagonal are zeros and not displayed. Each diagonal block $J_i$ is associated with an eigenvalue $\lambda_i$ of $A$ and the eigenvalues and blocks are not necessarily distinct. These Jordan blocks are $J_i= \left[ \begin{smallmatrix} \lambda_i & 1  & & \\ & \cdot & \cdot &  \\  & & \cdot &  1 \\ & & & \lambda_i \end{smallmatrix} \right]$\normalsize, where again the entries not displayed are zeros. Since the $A$-invariant subspaces \N\ are related to the $J$-invariant subspaces $\mathcal{M}$ according to $Q\mathcal{M}= \N $, and since $M$ commutes with $J$ when $QMQ^{-1}$ commutes with $A$, we may assume $A$ is in Jordan form $J$.\\[1ex]

\dfn A matrix $(t_{pq})$ is called a Toeplitz matrix if all of its entries satisfy $t_{p,q}=t_{p+1,q+1}$.     For $j\geq k$, a $j\by k$ matrix $U$,
is called an \emph{upper-triangular Toeplitz matrix} if $U$ is a Toeplitz matrix and $u_{p,q}=0$ for $p>q$. 
Similarly, if $j<k$,  the $j\by k$ matrix $U$, is called an \emph{upper-triangular Toeplitz matrix} if $U$ is a Toeplitz matrix and $u_{p,q}=0$ for $q-p<k-j$.  Thus, a non-square upper-triangular Toeplitz matrix has rows of zeros at the bottom if $j>k$ and  columns of zeros at the left if $k>j$. \\ [.65ex]

The matrices in $ \{J\}^\prime$ are known. \vspace{-2.3pt}

\begin{lemma}\label{commcond}   \cite[p. 297]{Gohberg} Let $J=\mbox{diag}[J_1,\ldots , J_u]$ be an $n \times n$ Jordan matrix with Jordan blocks $J_1,\ldots , J_u$ and eigenvalues $\lambda_1,\ldots , \lambda_u$, respectively and where $J_{\alpha}\,(\alpha = 1, \ldots, u)$ is a Jordan block of size $m_{\alpha}\times m_{\alpha}$.  Then an $n \times n$  matrix $M=[M_{\alpha \beta}]_{\alpha, \beta = 1}^u$ (blocked in the same partition as $J$ so that $M_{\alpha \beta}$ is an $m_{\alpha} \times m_{\beta}$ matrix) commutes with $ J $ if and only if $M_{\alpha \beta} = 0$ for $\lambda_\alpha \ne \lambda_\beta$, and $M_{\alpha \beta}$ is an upper-triangular Toeplitz matrix for $\lambda_\alpha = \lambda_\beta$. 
\end{lemma}

\begin{example}{A.} $J=$\tiny$ \left[ \begin{smallmatrix} 2 & 1 & 0&| &0 &|&0\\ 0  & 2 & 1 &| & 0&| &0\\ 0 & 0 & 2&| & 0&| &0\\ - & - & -&| & -&| &-\\ 0 & 0 & 0 &|& 2 &|&0 \\ - & - & -&| & -&| &-\\ 0 & 0 & 0 &|& 0 &|&3  \end{smallmatrix} \right]$\normalsize\ has what can obviously be called a 3-1-1 Jordan structure with eigenvalue 2 in the first two Jordan blocks 
and eigenvalue 3 in the third. Lemma 1 then says  $M \in \{J\}^\prime$ iff $M = $\tiny$\left[ \begin{smallmatrix}  a & b & c&| &d &|&0\\ 0  & a & b &| & 0&| &0\\ 0 & 0 & a&| & 0&| &0\\ - & - & -&| & -&| &-\\ 0 & 0 & e &|& f &|&0 \\ - & - & -&| & -&| &-\\ 0 & 0 & 0 &|& 0 &|&g  \end{smallmatrix} \right]$\normalsize\ with entries in $\mathbb{C}$. \hfill $\square$
\end{example}
\vspace*{6pt}

\noindent {\bf Notation.}\ In all that follows, $\vec{e}_{jk}$ will be the ``elementary basis vector'' filled with zeros except with 1 in the $j$th block's $k$th position. For example, for the 3-1-1 matrix $J$ in Example A, they are:\begin{center} \small$\vec{e}_{11}= \left[ \begin{smallmatrix} 1 \\ 0\\0 \\- \\ 0\\-\\ 0  \end{smallmatrix} \right]$, $\vec{e}_{12}= \left[ \begin{smallmatrix} 0 \\ 1\\0 \\- \\ 0\\-\\ 0  \end{smallmatrix} \right]$, $\vec{e}_{13}= \left[ \begin{smallmatrix} 0 \\ 0\\1 \\- \\ 0\\-\\ 0  \end{smallmatrix} \right]$, $\vec{e}_{21}= \left[ \begin{smallmatrix} 0 \\ 0\\0 \\- \\ 1\\-\\ 0  \end{smallmatrix} \right]$, and $\vec{e}_{31}= \left[ \begin{smallmatrix} 0 \\ 0\\0 \\- \\ 0\\-\\ 1  \end{smallmatrix} \right]$.\normalsize  \end{center}  Note the eigenvectors are $\vec{e}_{11}, \vec{e}_{21}$, and $\vec{e}_{31}$. Such an elementary ``Jordan basis'' is not unique, but the choice of $Q$ in $A=QJQ^{-1}$ produces it.

\vspace*{9pt}

\noindent {\bf $\S$ 1.} \textsc{The Case for \M  \ with Extremely Simple Form }  \\[1ex]

\noindent If $J$ has only one eigenvalue $\lambda$, then $J$ and $J-\lambda I$ have the same invariant subspaces.\ So without loss of generality, assume in this situation that $\lambda=0$ and $J$ is nilpotent. The next theorem and corollary are the versions of Theorem A for extremely simple forms of invariant subspaces \M.\\[1ex]

\begin{thm}\label{single}  If $J$ is a single $n \times n$ Jordan block, then every $J$-invariant subspace \M\ is  $\M=\ker (S^*)^k$, where $(S^*)^k$ is a determined power of the ``finite-dimensional backward shift,'' and $(S^*)^k \in \{ J \} ^\prime$.  
\end{thm}
\begin{proof} For such $J$, every $J$-invariant subspace \M\ has the form $\M= \mbox{span}\{\vec{e}_1, \ldots ,\vec{e}_k \}$, where $1 \le k \le n$ with eigenvector $\vec{e}_1$ and generalized eigenvectors $\vec{e}_j$, $j=2,3 \ldots n$ (cf. \cite[p. 7]{Gohberg}).  For this \M, construct the upper-triangular Toeplitz matrix $(S^*)^k$ having all zero diagonals except for the $(k+1)$st, which is $1$.\ (The main diagonal has $k=0$ and corresponds to the trivial subspace $\{ \vec{0}\}$.)  In a visual display, 
$(S^*)^k= \left[ \begin{smallmatrix}  0 & \cdots    & 0 & 1  &   \vspace{-5pt}\\ 
 &   &  &   &    &   \ddots    &    \\  
&     &  &    &     &    & 1  \\  
&    &  &  &   &  & 0 \vspace{-5pt} \\ 
 &    &  &  &   &  &  \vdots \\
 &    &  &  &   &  &  0 \end{smallmatrix} \right]$, 
with   blank entries filled with zeros.  Lemma 1 says $(S^*)^k \in\{ J \}^\prime$.\  Clearly $(S^*)^k   \vec{e}_j = \vec{0}$ for $j=1,2,\ldots, k$, and $\mbox{Rank }(S^*)^k =n-k$.\  Hence $\ker (S^*)^k  = \M$. \end{proof}
 
\begin{example}{B.} Let $J= \left[ \begin{smallmatrix} 0 & 1 & 0 \\ 0  & 0 & 1  \\ 0 & 0 & 0 \end{smallmatrix} \right]$.  $Z \in \{J\}^\prime$ iff $Z = \left[ \begin{smallmatrix}  a & b & c \\ 0  & a & b  \\ 0 & 0 & a   \end{smallmatrix} \right]$. The $J$-invariant subspaces are $\M_0=\{ \mathbf{0}\}$, $\M_1=\mbox{span}\{\vec{e}_1\}$, $\M_2=\mbox{span}\{\vec{e}_1, \vec{e}_2 \}$, and $\M_3=\mathbb{C}^3$. Theorem 2 constructs $\M_0=\mbox{ker }(S^*)^0=\ker  \left[ \begin{smallmatrix} 1 & 0 & 0 \\ 0  & 1 & 0  \\ 0 & 0 & 1 \end{smallmatrix} \right]$, $\M_1 = \ker S^*= \ker \left[ \begin{smallmatrix} 0 & 1 & 0 \\ 0  & 0 & 1  \\ 0 & 0 & 0 \end{smallmatrix} \right]$, $\M_2 = \ker (S^*)^2=\ker  \left[ \begin{smallmatrix} 0 & 0 & 1 \\ 0  & 0 & 0  \\ 0 & 0 & 0 \end{smallmatrix} \right]$, and $\M_3 = \ker (S^*)^3=\ker  \left[ \begin{smallmatrix} 0 & 0 & 0 \\ 0  & 0 & 0  \\ 0 & 0 & 0 \end{smallmatrix} \right]$. \hfill $\square$
\end{example}

The following corollary deals with the case in which the invariant subspace \M\ is a direct sum of subspaces covered by Theorem~\ref{single}. 
\begin{cor}\label{dsumsingle} 
If a nilpotent matrix $J$ has $p$ blocks indexed $1\le i \le p$, each with elementary basis eigenvector $\vec{e}_{i1}$ and size $n_i$, then a $J$-invariant subspace of the form \begin{center} $\M= \bigoplus\limits_{i=1}^p \mbox{span}\{ { \vec{e}_{i1},\ldots , \vec{e}_{ik_i} }\}$, where $k_i \le n_i$ \end{center}
can be realized as $\M=\ker Z$, where $Z$ is the diagonal block matrix of the same block structure as $J$ and whose $i$th diagonal block is $(S^*)^{k_i}$.  Furthermore, $Z\in \{ J \} ^\prime$ because it is of the matrix format of Lemma 1. 
\end{cor}
\begin{proof} Apply Theorem~\ref{single} to each of the direct sum components. The result follows.
\end{proof}

\noindent {\bf $\S$ 2.} \textsc{The Case Where All Jordan Blocks for $J$ Have the Same Eigenvalue} \\[1ex]

\noindent Again without loss of generality $J$ is nilpotent and has $p$ blocks indexed $1\le i \le p$, each with elementary basis eigenvector $\vec{e}_{i1}$ and size $n_i$. It is convenient to organize the Jordan blocks from largest to smallest moving from left to right across the matrix $J$.  Consider any ``irreducible'' invariant subspace (one that cannot be decomposed into a direct sum of multiple invariant subspaces); it has form  $\M= \{ \vec{v}, J\vec{v}, J^2\vec{v}, \ldots , J^{k-1}\vec{v},\vec{0} \}$, so that $J^{k-1}\vec{v}$ is an eigenvector.  {  An important distinction is that either \M\ is marked or it is not (again, see \cite[p. 210]{Rodman}).    For any marked subspace,} form the ``change of basis (transformation) matrix'' $T$, blocked the same way as $J$, in the following way:\vspace{-3pt}

\begin{itemize}[itemsep=0mm, topsep=1ex]
\item Write $J^{k-1}\vec{v}=\sum a_i \vec{e}_{i1}$ as a linear combination (with nonzero coefficients) of elementary basis eigenvectors, then identify the rightmost Jordan block of $J$ that has its eigenvector in this linear combination. (The irreducible \M\ will have chain length no more than the dimension of this block.) 
\item For that (say it is the $j$th) diagonal block's corresponding columns of $T$,   let the first $k$  columns be $J^{k-1}\vec{v}, \ldots, J^2\vec{v}, J\vec{v}, \vec{v}$.\ Fill out any remaining columns of that block of $T$ with columns that extend any of those $k$ column's nonzero entries in an upper-triangular Toeplitz manner and have zeros elsewhere.  
\item Set the remaining blocks of $T$ equal to zero but put the identity on all other diagonal blocks.  
\end{itemize}
Then construct $Z$ as a block diagonal matrix with the identity in each diagonal block except for the $j^{th}$ block, which is $(S^*)^k$.\ This construction of $T$ and $Z$ produces the following version of Theorem A for this scenario:

\begin{theorem}{A1.}
{\it Given an irreducible marked invariant subspace \M\ for a Jordan matrix $J$ with a single eigenvalue, the construction of the $n \times n$ matrices $T$ and $Z$ described immediately above produces \begin{center}$\boxed{\, M=ZT^{-1} \mbox{ with } \ \M\ = \ker M\, }$\end{center}}
\end{theorem}

\begin{proof} $T$ has the following properties: 
\begin{itemize}[itemsep=0mm, topsep=1ex]
  \item $T$ sends  {  $\vec{e}_{j1}, \ldots , \vec{e}_{jn_j}$ to the vectors that form its columns running through the $j$th diagonal block.  In particular, it sends  each of } $\vec{e}_{j1}, \ldots , \vec{e}_{j(k-1)} , \vec{e}_{jk}$ to  $J^{k-1}\vec{v}, \ldots, J^2\vec{v}, J\vec{v}, \vec{v}$, respectively. For any other elementary basis vector $\vec{e}$ not already discussed in this bullet, $T\vec{e} = \vec{e}$.   
\item $T$ is invertible, since (see \cite[p.183]{Axler}) the columns of $T$ form a basis for 
$\mathbb{C}^n$ and so are linearly independent.
\item $T^{-1}$ forms a new coordinate system.  The only changes to the elementary basis vectors are to   $\vec{e}_{j1}, \ldots , \vec{e}_{jn_j}$, which are transformed to the new Jordan chain coordinate basis vectors running through the $j$th diagonal block (and include $J^{k-1}\vec{v}, \ldots, J^2\vec{v}, J\vec{v}, \vec{v}$).
\item $T$ is in $\{J\}^\prime$, since it is in the form described in Lemma 1 -- the nature of the Jordan chain structure produces an upper-triangular Toeplitz system in each block.   In fact, an easy way to check that \M\ is   marked\footnote{See \cite[ p. 84, Theorem 2.9.1]{Gohberg} for an equivalent procedure in this situation.} is that $T$ is in $\{J\}^\prime$. Any instance where this does not happen is remedied in the non-marked case described below.   
\end{itemize} 
\noindent The new coordinate system makes the subspace \M$_T=T^{-1}$\M\ of the simple type in Corollary 3, which then constructs $Z$ as described for this scenario.  By Corollary 3, \M$_T=\ker Z$, which gives the desired representation of Theorem A with $P$ the identity: \begin{center}\M\ $=T$\M$_T=T\ker Z = \ker (ZT^{-1})$.\end{center}
\vspace*{-24pt}
\end{proof}

For the situation in Theorem A1, note that each different $Z$ can be paired with each \M, up to the number of elementary basis vectors forming the chain basis for \M, since $Z$ is constructed exactly from that number.  This observation will provide a natural way to categorize different invariant subspaces; we say that two invariant subspaces are of the same type when their chain basis vectors from each Jordan block components have the same length.  Each different subspace of a given type is described simply from the coefficients in the linear combination of elementary basis vectors. Section 3 will use this fact to construct invariant subspace lattices.\\[1ex]

{  The situation is more difficult when \M\ is not marked; the construction is similar but requires amendment. Because, for example, any ${J}$-invariant subspace must be marked when ${J}$ has Jordan blocks whose sizes differ by at most one (cf. \cite[p. 211]{Rodman}), we expand the matrix $J$ -- we choose to expand individual Jordan blocks until the formation of $T$ will satisfy the properties listed above. This expanding of $J$ to a new nilpotent Jordan block matrix $\hat{J}$ for which \M\ will be marked can always be performed by bringing block sizes toward equality. In the expanded space, formulate $Z T^{-1}$, with $T$ constructed for $\hat{J}$ as in Theorem A1, but use the following amended construction for $Z$.\\[1ex]

\noindent Suppose the expanding of $J$ to $\hat{J}$ has added $p$ rows and columns to the $i$th diagonal block of $J$.  Then:
\begin{itemize}[itemsep=0mm, topsep=1ex]
  \item Make each diagonal block of $Z$ the identity with two exceptions: make the $j$th diagonal block $(S^*)^k$, and make the $i$th diagonal block $(S^*)^{p}$.
  \item Make each off-diagonal block zero, with two exceptions: make the $Z_{j{i}}$ block the identity (if the block is not square, then add zeros on the left  or below as needed to fill out the block), and make the $Z_{{i}j}$ block  $(S^*)^{p}$ (again with zeros added on the left  or below if the block is not square).
\end{itemize} 
\noindent Note $ZT^{-1}$ is an $(n+p) \times (n+p)$ matrix in this scenario. To shrink the construction back down to size $n \times n$, use an $n \times (n+p)$ row selection matrix  $P$ that has all zero entries except for a single 1 on each row and located on increasing numbered columns that  correspond to an original row of $J$. The effect of simultaneous multiplication on the left by $P$ and on the right by $P^t$ is simply to remove each of the rows and columns that were added in the expanding. This construction of $T$, $Z$, and $P$ produces the following version of Theorem A for this scenario:

\begin{theorem}{A2.}
{\it Given an irreducible nonmarked invariant subspace \M\ for a Jordan matrix $J$ with a single eigenvalue, the construction of the $(n+p) \times (n+p)$ matrices $T$ and $Z$, and the $n \times (n+p)$ matrix $P$ described immediately above produces the $n \times n$ matrix $M$ so that \begin{center}$\boxed{\, M=PZT^{-1}P^t \mbox{ with } \ \M\ = \ker M \, }$\end{center}}
\end{theorem}

\begin{proof}  $M=P Z  T^{-1}P^t$ has the two desired properties: 
$ \mbox{(i) } \M = \ker M \mbox{ and (ii) } P Z  T^{-1}P^t \in \{ J  \}^\prime$.
  
\noindent The first fact follows from $ T^{-1} T=I$, and so $T^{-1}$ acting on the array of columns that form $P^t$\M\ (and hence form certain columns of $T$) produce columns that have blocks of zeros except for the $j$th block being the $n_j \times n_j$ identity.  When then multiplied by $Z$, the first $k$ columns of this identity are sent to zero (by the appropriate power of the backward shift), and so $P Z T^{-1}P^t$ sends all entries in the array formed from the columns of vectors from \M\ to zero. The second fact follows because  $(S^*)^{p}$ applied to upper-triangular blocks -- this occurs in the matrix multiplication $Z T^{-1}$ --sends enough lower left elements of the $ji$ and $ij$ blocks of  $ZT^{-1}$  to zero  to insure $ZT^{-1} \in \{ \hat{J}\}^{\prime}$, and hence $ P Z  T^{-1}P^t \in \{ J  \}^\prime$.  
\end{proof}

Finally, consider the situation that \M\ is reducible; i.e., it can be written as a direct sum of more than one nonzero irreducible (chain) subspaces.  The construction easily modifies according to the irreducible subspaces in the direct sum, thinking of $T$ and $Z$ as being blocked in a corresponding manner to these pieces. Start with the irreducible piece that has the shortest chain and, if there is more than one of those  pieces with the same length, start with the one that has in its eigenvector's direct sum the elementary row eigenvector of the right-most block. Construct $T$ in this block's columns as before, then move to the next such shortest chain's corresponding block, and continue until each irreducible piece has the corresponding formation in $T$. Because the process starts with the shortest chain in the shortest block, there is always room at each step to fill in the chain's vectors as columns to form $T$.  After using all of the irreducible piece's chain vectors in the construction of $T$, fill out the rest of $T$ as before, with identities on the other diagonal pieces. Construct $Z$ in an analogous manner, using the previous construction in the corresponding portions of $Z$ for each irreducible piece. In these constructions of $Z$ and $T$, expand to $\hat{J}$ as before when any one of the irreducible pieces is nonmarked, using $P^t$ and $P$ to affect the expansion and shrinking from and to size $n \times n$ as before, and  if each irreducible piece of \M\ is marked, then $P$ is the identity. Theorem A then results for this case.  Example C will give a simple illustration for its subspace $\M_3$. 

\begin{theorem}{A3.}
{\it Given a reducible invariant subspace \M\ for a Jordan matrix $J$ with a single eigenvalue, construct matrices $T$, $Z$, and $P$ as above and in terms of each irreducible piece of \M\ to produce the $n \times n$ matrix $M$ with $ M=PZT^{-1}P^t \mbox{ and } \ \M\ = \ker M  $.}
\end{theorem}

\begin{proof} 
Theorems A1 and A2 show the construction produces  $\M_n \subseteq \ker M$ for each irreducible piece $\M_n$ that forms the direct sum of \M. Taken together in the direct sum, this produces $\M=\ker M$. As before, the matrix $M$ is in the Lemma 1 form of matrices in $\{ J \}^\prime$. 
\end{proof} 

\begin{example}{C.} {\bf $J \sim (3-2-1)$}. The Jordan blocks are $3\times 3$, $2\times 2$, and $1 \times 1$. Consider the following illustrative subspaces, listed using general scalars  $A, B, C, D \in \mathbb{C}$:
\begin{itemize}
  \item  $\M_{1}=\mbox{span} \{ \vec{e}_{11}+A\vec{e}_{21}, \vec{e}_{12}+A\vec{e}_{22}+B\vec{e}_{11}+C\vec{e}_{21}\} $,  
\item  $\M_{2}= \mbox{span} \{ \vec{e}_{11}, \vec{e}_{12}+A\vec{e}_{11} + \vec{e}_{31} \}$, and 
\item $\M_{3}= \mbox{span} \{ \vec{e}_{11}+A\vec{e}_{21}, \vec{e}_{12}+A\vec{e}_{22}+B\vec{e}_{11}+C\vec{e}_{21} \} \bigoplus \mbox{span} \{  \vec{e}_{31}+D\vec{e}_{21} \} $.
\end{itemize}

\noindent $\bullet$ For $\M_{1}$, $T=$\tiny$\left[ \begin{smallmatrix} 1&0&0&|&1 & B  & | & 0 \\  0 & 1 & 0 & | & 0 & 1&|&0 \\ 0 & 0 & 1&| &0&0& | & 0  \\ -&-& -& - &-& -& - &- \\0& 0 &0 &| & A &C&|&0 \\ 0 & 0 & 0&| &0&A& | & 0  \\ -&-& -& - &-& -& - &- \\0& 0 &0 &| & 0 &0&|&1 \end{smallmatrix} \right]$\normalsize\ and   $Z=$\tiny$\left[ \begin{smallmatrix} (S^*)^0&| &  [0]&| & [0] \\-& -& -& - &- \\  [0] & | & (S^*)^2&| & [0]\\  -& -& -& - &- \\ [0] & | &  [0] &| & (S^*)^0 \end{smallmatrix} \right]=\left[ \begin{smallmatrix} 1 & 0 & 0&| &  0 & 0&| & 0 \\  0 & 1 & 0&| &  0 & 0&| & 0\\  0 & 0 &1& | &  0 & 0&| & 0\\-& -& -& - &- & -& - &- \\0& 0 &0&| &  0 & 0&| & 0 \\  0 & 0 &0& | &  0 & 0&| & 0\\-& -& -& - &- & -& - &- \\0& 0 &0&| &  0 & 0&| & 1 \end{smallmatrix} \right] $.\normalsize\  \\
\noindent $\M_{1}=\ker( ZT^{-1})= \ker$\tiny$ \left[ \begin{smallmatrix} 
1&  0&  0& |&  -1/A&  (C-AB )/A^2& |&  0\\ 0&  1&  0& |&  0&  -1/A& |&  0\\ 0&  0&  1& |&  0&  0& |&  0\\-& -& -& - &- & -& - &- \\ 0&  0&  0& |&  0&  0& |&  0\\ 0&  0&  0& |&  0&  0& |&  0\\-& -& -& - &- & -& - &- \\ 0&  0&  0& |&  0&  0& |&  1    \end{smallmatrix} \right]       $.\normalsize\ \\[1ex]

\noindent  $\bullet \ \M_{2}$ is irreducible but not marked. expand the 3rd Jordan block by one.\\ \noindent  $T=$\tiny$\left[ \begin{smallmatrix} 1&A&0&|&0 & 0  & | & 0 &0 \\  0 & 1 & A & | & 0 & 0&|&0 &0 \\ 0 & 0 & 1&| &0&0& | & 0 &0  \\ -&-& -& - &-& -& - &- &- \\0& 0 &0 &| & 1 &0&|&0 &0 \\ 0 & 0 & 0&| &0&1& | & 0 &0  \\ -&-& -& - &-& -& - &- &- \\0& 1 &0 &| & 0 &0&|&1 &0 \\0& 0 &1 &| & 0 &0&|&0 &1 \end{smallmatrix} \right]$,\normalsize\  
\noindent $Z=$\tiny$\left[ \begin{smallmatrix}  (S^*)^2 &| &  [0] \ &| &  I_2\\ &| & &| & [0] \\-& -& -& - &- \\  [0] & | & (S^*)^0&| & [0]\\  -& -& -& - &- \\ [0] \ (S^*)^1 & | &  [0] &| & (S^*)^1 \end{smallmatrix} \right]=\left[ \begin{smallmatrix} 0 & 0 & 1&| &  0 & 0&| & 1  & 0 \\  0 & 0 & 0&| &  0 & 0&| & 0  & 1 \\  0 & 0 &0& | &  0 & 0&| & 0  & 0 \\-& -& -& - &- & -& - &-  & -\\0& 0 &0&| &  1 & 0&| & 0  & 0  \\  0 & 0 &0& | &  0 & 1&| & 0  & 0 \\-& -& -& - &- & -& - &-  & - \\0& 0 &1&| &  0 & 0&| & 0  & 1  \\0& 0 &0&| &  0 & 0&| & 0  & 0 \end{smallmatrix} \right] $,\normalsize\ $P=$\tiny$\left[ \begin{smallmatrix} 1 & 0 & 0&| &  0 & 0&| & 0  & 0 \\  0 & 1 & 0&| &  0 & 0&| & 0  & 0 \\  0 & 0 &1& | &  0 & 0&| & 0  & 0 \\-& -& -& - &- & -& - &-  & -\\0& 0 &0&| &  1 & 0&| & 0  & 0  \\  0 & 0 &0& | &  0 & 1&| & 0  & 0 \\-& -& -& - &- & -& - &-  & - \\0& 0 &0&| &  0 & 0&| & 1  & 0   \end{smallmatrix} \right] $.\normalsize\ \\
\noindent $\M_{2}=\ker(P Z T^{-1}P^t)=\ker \left( P\right.$\tiny$ \left[ \begin{smallmatrix}   0&  -1&  1 +A &  0&  0&  1&  0\\ 0&  0&  -1&  0&  0&  0&  1\\ 0&  0&  0&  0&  0&  0&  0\\ 0&  
0&  0&  1&  0&  0&  0\\ 0&  0&  0&  0&  1&  0&  0\\ 0&  0&  0&  0&  0&  0&  1\\ 0&  0&  0&  0&  0&  0&  0  \end{smallmatrix} \right]$\normalsize$ P^t\left.\right) =\ker$\tiny$   \left[ \begin{smallmatrix}   0&  -1&  1 +A &|&  0&  0& |& 1 \\ 0&  0&  -1&|&  0&  0&|&  0 \\ 0&  0&  0&|&  0&  0&|&  0 \\-& -& -& - &- & -& - &-\\ 0& 0&  0&|&  1&  0&|&  0 \\ 0&  0&  0& |& 0&  1& |& 0\\-& -& -& - &- & -& - &- \\ 0&  0&  0& |& 0&  0& |& 0    \end{smallmatrix} \right].$\normalsize\  \\


\noindent  $\bullet \ \M_{3}$ is reducible. $T=$\tiny$\left[ \begin{smallmatrix} 1&0&0&|&1 & B  & | & 0 \\  0 & 1 & 0 & | & 0 & 1&|&0 \\ 0 & 0 & 1&| &0&0& | & 0  \\ -&-& -& - &-& -& - &- \\0& 0 &0 &| & A &C&|&D \\ 0 & 0 & 0&| &0&A& | & 0  \\ -&-& -& - &-& -& - &- \\0& 0 &0 &| & 0 &0&|&1 \end{smallmatrix} \right]$,\normalsize\  $Z=$\tiny$\left[ \begin{smallmatrix} (S^*)^0&| &  [0]&| & [0] \\-& -& -& - &- \\  [0] & | & (S^*)^2&| & [0]\\  -& -& -& - &- \\ [0] & | &  [0] &| & (S^*)^1 \end{smallmatrix} \right]=\left[ \begin{smallmatrix} 1 & 0 & 0&| &  0 & 0&| & 0 \\  0 & 1 & 0&| &  0 & 0&| & 0\\  0 & 0 &1& | &  0 & 0&| & 0\\-& -& -& - &- & -& - &- \\0& 0 &0&| &  0 & 0&| & 0 \\  0 & 0 &0& | &  0 & 0&| & 0\\-& -& -& - &- & -& - &- \\0& 0 &0&| &  0 & 0&| & 0 \end{smallmatrix} \right] $,\normalsize\ and $P$ is the identity.\ (Each piece is marked).\  
$\M_{3}=\ker( ZT^{-1})= \ker$\tiny$ \left[ \begin{smallmatrix} 
1&  0&  0& |&  -1/A&  (C-AB )/A^2& |&  D/A\\ 0&  1&  0& |&  0&  -1/A& |&  0\\ 0&  0&  1& |&  0&  0& |&  0\\-& -& -& - &- & -& - &- \\ 0&  0&  0& |&  0&  0& |&  0\\ 0&  0&  0& |&  0&  0& |&  0\\-& -& -& - &- & -& - &- \\ 0&  0&  0& |&  0&  0& |&  0   \end{smallmatrix} \right]       $.\normalsize\ \hfill $\square$
\end{example}

\vspace*{9pt}

\noindent {\bf $\S$ 3.} \textsc{The Construction Quickly Determines the Invariant Subspace Lattice}  \\[1ex]

\noindent Another way to handle the construction of $Z$ for subspaces that are not marked provides a remarkably simple way to formulate the complete lattice of $J$-invariant subspaces.  Any $J$-invariant subspace \M\ is also $\hat{J}$-invariant, where $\hat{J}$ is identical to $J$ except for individual Jordan blocks that may be expanded.  Hence any $J$-invariant non-marked \M\ can always be considered as $\hat{J}$-invariant marked for an expanded $\hat{J}$. We can easily construct $Z$ as in Theorem A1 for $\hat{J}$, without worrying what it is for $J$, and that construction identifies the subspace \M\ in the expanded structure.  This idea quickly produces the invariant subspace lattice for $J$ as in the following proof of Theorem B and illustrated in Example D.

\begin{proof}[Proof of Theorem B] Every $J$-invariant subspace \M\ can be considered marked -- if not for $J$ then for $\hat{J}$, where $\hat{J}$ is an expansion of $J$ upward so that the difference in the sizes between any Jordan block is at most 1. Then every $J$-invariant subspace corresponds in a pairwise fashion to a block-diagonal $Z$ as constructed to produce Theorem A1, either for $J$ or an expanded $\hat{J}$.\  The lattice of the matrices for $Z$ is then easy to construct:\ start with the $n \times n$ identity at the bottom of the lattice, which is thought of as the diagonal blocks of backward shifts taken to the 0{\it th} power.  Then simply raise any chosen diagonal block's backward shift power by one to get a new $Z$ at the next higher level in the lattice! The sum of the powers on the backward shifts equals the dimension of the subspace.\ Take care to include $Z$ for any subspace that is not marked; whenever the sizes of any two Jordan blocks differ by more than one, expand the dimension  of a smaller block up to any size $t$ that is one less than the larger block size to obtain a marked subspace as described above. The subspace becomes marked in the expanded space, and the expanded (now diagonal block) $Z$ fits properly into the lattice.
\end{proof}

\begin{example}{D.} {\bf $J \sim (3-1)$}; i.e. the Jordan blocks of $J$ are of size $3\times 3$ and $1 \times 1$. There are seven types of nontrivial invariant subspaces, listed here with general scalars  $A, B \in \mathbb{C}$.  The dimension 3 subspaces are \vspace{-20pt}

\begin{center} 
  \item  $\M_{31}=\mbox{span} \{ \vec{e}_{11}, \vec{e}_{12}+A\vec{e}_{11},\vec{e}_{13}+A\vec{e}_{12}+B\vec{e}_{11}\} $ and    $\M_{32}= \mbox{span} \{ \vec{e}_{11}, \vec{e}_{12}+A\vec{e}_{11} \} \bigoplus \mbox{span} \{  \vec{e}_{21} \} $.
\end{center}

\noindent The dimension 2 subspaces are
\begin{center} 
   $\M_{21}=\mbox{span} \{ \vec{e}_{11}, \vec{e}_{12} +A\vec{e}_{11}\} $,
    $\M_{22}=\mbox{span} \{ \vec{e}_{11}, \vec{e}_{12} +A\vec{e}_{11}+B \vec{e}_{21}\} $ (which is not marked), and  
$\M_{23}= \mbox{span} \{ \vec{e}_{11}, \vec{e}_{21} \} = \mbox{span} \{ \vec{e}_{11} \} \bigoplus \mbox{span} \{  \vec{e}_{21} \} $.
\end{center}

\noindent The dimension one subspaces are
\begin{center}
$\M_{11}= \mbox{span} \{ \vec{e}_{11} \}$, and
 $ \M_{12}= \mbox{span} \{ A\vec{e}_{11} +\vec{e}_{21}\}  $. 
\end{center}
Each subspace has been written as the direct sum of irreducible subspaces.   By the construction for Theorem A1, {\it where non-marked \M\ are considered in an expanded space as marked so that $Z$ is block diagonal}, the matrix $Z$ is as follows for each subspace.

\begin{multicols}{4}
\noindent $\left[ \begin{smallmatrix} 0 & 0 & 0&| & 0 \\  0 & 0 & 0&| & 0 \\  0 & 0 &0& | & 0 \\-& -& -& - &- \\0& 0 &0&| & 1 \end{smallmatrix} \right] $ for $\M_{31}$\vspace{3pt}

\noindent $\left[ \begin{smallmatrix} 0 & 0 & 1&| & 0 \\  0 & 0 & 0&| & 0 \\  0 & 0 &0& | & 0 \\-& -& -& - &- \\0& 0 &0&| & 0 \end{smallmatrix} \right] $ for $\M_{32}$\vspace{3pt}

\noindent $\left[ \begin{smallmatrix} 0 & 0 & 1&| & 0 \\  0 & 0 & 0&| & 0 \\  0 & 0 &0& | & 0 \\-& -& -& - &- \\0& 0 &0&| & 1 \end{smallmatrix} \right] $ for $\M_{21}$\vspace{3pt}

\noindent $\left[ \begin{smallmatrix} 0 & 0 & 1&| & 0 &0\\  0 & 0 & 0&| & 0&0 \\  0 & 0 &0& | & 0&0 \\-&-& -& -& - &- \\0& 0 &0&| & 1 &0 \\0& 0 &0&| & 0 &1 \end{smallmatrix} \right] $ \small for \normalsize $\M_{22}$\vspace{3pt}

\noindent $\left[ \begin{smallmatrix} 0 & 1 & 0&| & 0 \\  0 & 0 & 1&| & 0 \\  0 & 0 &0& | & 0 \\-& -& -& - &- \\0& 0 &0&| & 0 \end{smallmatrix} \right] $ for $\M_{23}$\vspace{3pt}

\noindent $\left[ \begin{smallmatrix} 0 & 1 & 0&| & 0 \\  0 & 0 & 1&| & 0 \\  0 & 0 &0& | & 0 \\-& -& -& - &- \\0& 0 &0&| & 1 \end{smallmatrix} \right] $ for $\M_{11}$\vspace{3pt}

\noindent $\left[ \begin{smallmatrix} 1 & 0 & 0&| & 0 \\  0 & 1 & 0&| & 0 \\  0 & 0 &1& | & 0 \\-& -& -& - &- \\0& 0 &0&| & 0 \end{smallmatrix} \right] $ for $\M_{12}$
\end{multicols}

\noindent As described in the proof of Theorem B, a lattice with these $Z$ matrices is simple to construct.
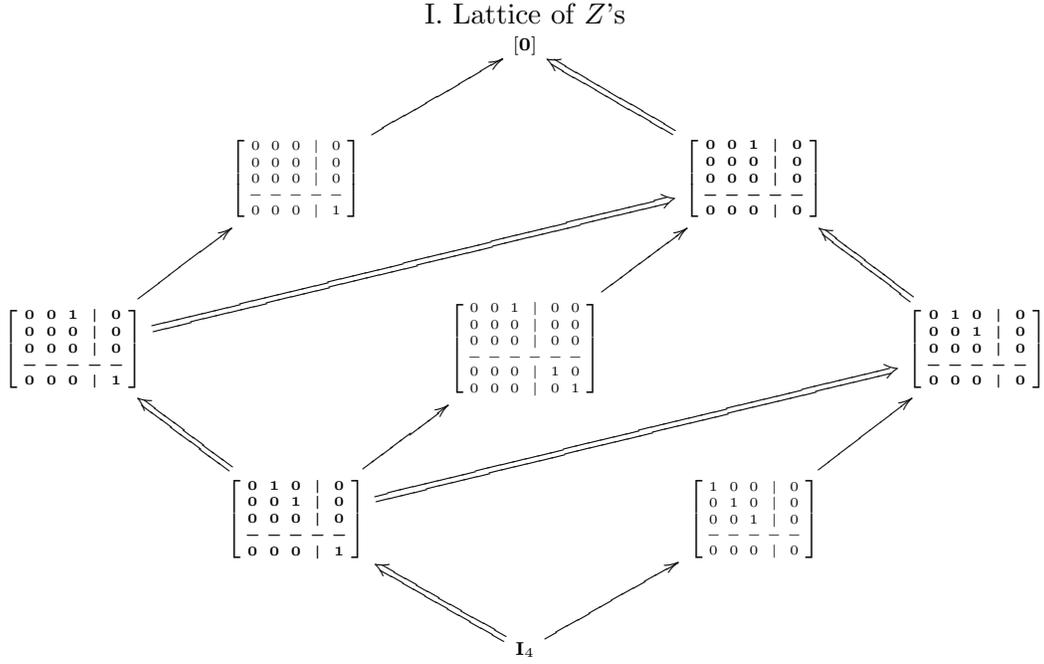
\begin{center} {I. Lattice of $Z$'s} \\
\tiny
$\xymatrix{ 
& & {\bf [0]}& &\\ & {\left[ \begin{smallmatrix} 0 & 0 & 0&| & 0 \\  0 & 0 & 0&| & 0 \\  0 & 0 &0& | & 0 \\-& -& -& - &- \\0& 0 &0&| & 1 \end{smallmatrix} \right] }  \ar[ur] &  & {{\boldsymbol{\left[ \begin{smallmatrix} 0 & 0 & 1&| & 0 \\  0 & 0 & 0&| & 0 \\  0 & 0 &0& | & 0 \\-& -& -& - &- \\0& 0 &0&| & 0 \end{smallmatrix} \right]}}} \ar@{=>}[ul] &\\
{{\boldsymbol{ \left[ \begin{smallmatrix} 0 & 0 & 1&| & 0 \\  0 & 0 & 0&| & 0 \\  0 & 0 &0& | & 0 \\-& -& -& - &- \\0& 0 &0&| & 1 \end{smallmatrix} \right]}}} \ar[ur] \ar@{=>}[urrr] &  & { \left[ \begin{smallmatrix} 0 & 0 & 1&| & 0 &0\\  0 & 0 & 0&| & 0&0 \\  0 & 0 &0& | & 0&0 \\-&-& -& -& - &- \\0& 0 &0&| & 1 &0 \\0& 0 &0&| & 0 &1 \end{smallmatrix} \right] }  \ar[ur]  &  & {{\boldsymbol{\left[ \begin{smallmatrix} 0 & 1 & 0&| & 0 \\  0 & 0 & 1&| & 0 \\  0 & 0 &0& | & 0 \\-& -& -& - &- \\0& 0 &0&| & 0 \end{smallmatrix} \right]}} } \ar@{=>}[ul] \\& {\boldsymbol{ \left[ \begin{smallmatrix} 0 & 1 & 0&| & 0 \\  0 & 0 & 1&| & 0 \\  0 & 0 &0& | & 0 \\-& -& -& - &- \\0& 0 &0&| & 1 \end{smallmatrix} \right] } }\ar@{=>}[ul]  \ar[ur]  \ar@{=>}[urrr]   & & {\left[ \begin{smallmatrix} 1 & 0 & 0&| & 0 \\  0 & 1 & 0&| & 0 \\  0 & 0 &1& | & 0 \\-& -& -& - &- \\0& 0 &0&| & 0 \end{smallmatrix} \right]}  \ar[ur] &\\
& & {\mathbf I_4}  \ar[ur] \ar@{=>}[ul]& &
 }$
\end{center}\normalsize 
The correspondence with the invariant subspace lattice  $\L(A)$ is easy to see.
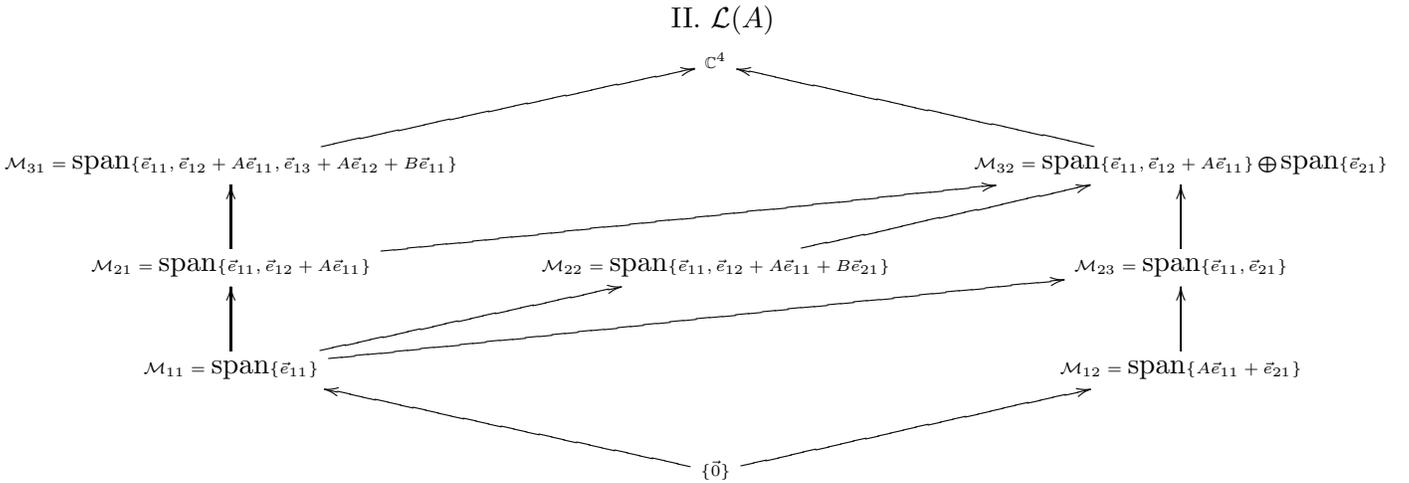
\begin{center} 
 { II. $\L(A)$}\\
\fontsize{6}{6}
\hspace*{-35pt}
$\xymatrix{ 
&\mathbb{C}^4& \\ \M_{31}=\mbox{span} \{ \vec{e}_{11}, \vec{e}_{12}+A\vec{e}_{11},\vec{e}_{13}+A\vec{e}_{12}+B\vec{e}_{11}\}  \ar[ur] &  & \M_{32}= \mbox{span} \{ \vec{e}_{11}, \vec{e}_{12}+A\vec{e}_{11} \} \bigoplus \mbox{span} \{  \vec{e}_{21} \} \ar[ul]\\
\M_{21}=\mbox{span} \{ \vec{e}_{11}, \vec{e}_{12} +A\vec{e}_{11}\}  \ar[u] \ar[urr]   & \M_{22}=\mbox{span} \{ \vec{e}_{11}, \vec{e}_{12} +A\vec{e}_{11}+B \vec{e}_{21}\}  \ar[ur]    & \M_{23}=\mbox{span} \{ \vec{e}_{11}, \vec{e}_{21} \}  \ar[u] \\ \M_{11}=\mbox{span} \{ \vec{e}_{1 1}\}  
  \ar[u] \ar[ur]  \ar[urr]   & & \M_{12}=\mbox{span} \{ A\vec{e}_{11} +\vec{e}_{21}\} \ar[u] \\
& \{ \vec{0}\}  \ar[ur] \ar[ul]&
 }$
\end{center}

\normalsize 

The lattice of $Z$'s also quickly identifies the hyperinvariant subspaces -- those subspaces invariant not only for $J$ but also for all matrices in $\{J\}^\prime$.  They have been determined by Fillmore, et al in \cite[p. 128]{Fillmore} to  form a sublattice inside the full invariant subspace lattice.  Reformulating Fillmore's analysis to our $Z$ construction, any hyperinvariant subspace is one that is marked for $J$ and has corresponding $Z$ matrix whose individual diagonal blocks with sizes $n_1,\ldots, n_m$ (ordered from larger to smaller) have powers $r_1,\ldots,r_m$  on the backward shift that satisfy:\begin{itemize}[itemsep=-1mm]
\item  $r_1 \ge r_2 \ge \cdots \ge r_m \ge 0$, and
\item $n_1-r_1 \ge n_2-r_2 \ge \cdots \ge n_m-r_m \ge 0$.
\end{itemize}
Such $Z$ matrices, and hence the corresponding hyperinvariant subspaces, are simple to spot. They are in bold in the above lattice of $Z$'s, and the connecting lattice arrows are doubled to display the hyperinvariant subspace sublattice. 
\end{example}

\newpage
\noindent {\bf $\S$ 4.} \textsc{The Case for Any Jordan Form.}  \\[1ex]

\noindent The construction from Section 2 easily generalizes.

\begin{theorem}{A4.}  If $J$ is composed of $u$ (possibly more than one) Jordan blocks, say with a total of $m$ different eigenvalues $\lambda_i, i=1, \ldots , m$, then every $J$-invariant subspace \M\ is of the form \begin{center} $\M= \ker M= \ker(P Z  T^{-1}P^t) \equiv \ker  \left[ \begin{smallmatrix} P_1Z_1 T_1^{-1}P_{1}^t &  &  \vspace{-5pt}\\   & \ddots &   \\ & & P_m Z_m T_m^{-1} P_{m}^t \end{smallmatrix} \right]$\end{center} (blank portions of the display are filled with zeros), where each $P_i$, $Z_i$, and $T_i$ are constructed, one for each distinct eigenvalue, as in Theorem A3, and each $P$, $Z$, and $T^{-1}$ are the block diagonal matrices that each consists of its corresponding $m$ diagonal blocks indexed on $i$.  
\end{theorem}
 
\begin{proof}  Group together the Jordan blocks that share the same eigenvalue, making them adjacent in the Jordan form for $J$.  In this situation,  $\M=\M_1 \oplus \ldots \oplus \M_m $, where each $\M_i$ corresponds to a distinct one of the eigenvalues $\lambda_i$. The result follows immediately, applying the construction from Theorem A3 to each separate eigenvalue block and noting that the direct sum structure from different eigenvalue pieces fills the remaining portion of the matrix block structure of $M$ with zeros (a required condition for $Z$'s commutivity with $J$ implied by Lemma 1).  
\end{proof}

\begin{example}{E.} Examine $J=$\tiny$ \left[ \begin{smallmatrix} 2 & 1 & 0 & |&0 & \| &0\\ 0  & 2 & 1 & | & 0 & \| &0\\ 0 & 0 & 2 & |  & 0 & \| &0 \\ - & - & - & - &- & - & -  \\ 0 & 0 & 0 & | & 2 & \| &0 \\  = & = & = & = &= & = & = \\ 0 & 0 & 0 & |& 0& \| &3  \end{smallmatrix} \right]$,\normalsize\ the matrix in Example A.\ (Note blocks with common eigenvalue 2 are grouped together.) The subspace $\M=\mbox{span}\{\vec{e}_{11}, \vec{e}_{12}+A  \vec{e}_{21}+ B \vec{e}_{11} \}  \oplus  \mbox{span}\{  \vec{e}_{31} \}$,  where $A ,B    \in \mathbb{C}$, has non-marked first piece from the eigenvalue 2 and marked second piece from the eigenvalue 3. Then
\begin{center}  $T =$\tiny$  \left[ \begin{smallmatrix} 1   & B  & 0 & |& 0 &0 & \|  &0\\ 0  & 1   & B   & | & 0 &0  & \| &0\\ 0 & 0 & 1  & |  & 0 &0  & \| &0 \\ - & - & - & - &- & -&-&-  \\ 0 & A & 0  & | & 1  &0 & \| &0 \\ 0 & 0 & A  & | & 0  &1 & \| &0 \\ = & = & = & = &= & = & = &=\\ 0 & 0 & 0 & |& 0&0& \| &1  \end{smallmatrix} \right]$\normalsize,  $Z =$\tiny$  \left[ \begin{smallmatrix} (S^*)^2 & | & I_2 & \| & [0] \\ & | & [0] & \| &  \\ - & - & - & - & - \\ [0]\ (S^*)^1  & | & (S^*)^1 & \| & [0] \\ = & = & = & = & = \\ [0] & | & [0] & \| & (S^*)^1     \end{smallmatrix} \right]
   =  \left[ \begin{smallmatrix} 0   & 0  & 1 & |& 1 &0 & \|  &0 \\ 0  & 0   & 0   & | & 0 &1& \| &0 \\ 0 & 0 & 0  & |  & 0 &0& \| &0 \\ - & - & - & - &- &-  & -&-  \\ 0 & 0 & 1 & | & 0 & 1&\| &0  \\ 0 & 0 & 0  & | & 0&0 & \| &0  \\ =& = & = & = & = &= & = & = \\ 0 & 0 & 0 & |& 0& 0&\| &0  \end{smallmatrix} \right]$\normalsize,  $P=$\tiny$ \left[ \begin{smallmatrix} 1   & 0  & 0 & |& 0 &0 & \|  &0\\ 0  & 1   & 0   & | & 0 &0  & \| &0\\ 0 & 0 & 1  & |  & 0 &0  & \| &0 \\ - & - & - & - &- & -&-&-  \\ 0 & 0 & 0  & | & 1  &0 & \| &0  \\ = & = & = & = &= & = & = &=\\ 0 & 0 & 0 & |& 0&0& \| &1  \end{smallmatrix} \right]$\normalsize, and
  $\M = \ker ( P$\tiny$  \left[ \begin{smallmatrix} 0 & -A  & 1+AB & |&1  &0& \|  &0\\ 0  & 0 & -A  & | & 0 &1& \| &0\\ 0 & 0 & 0 & |  & 0&0& \| &0 \\ - & - & - & - &- & - & - &- \\ 0 & 0 & 1-A & | & 0 &1& \| &0 \\ 0 & 0 & 0 & | & 0 & 0 & \| &0 \\  = & = & = & = &= & = & = &=\\ 0 & 0 & 0 & |& 0&0& \| &0  \end{smallmatrix} \right]$\normalsize$P^t )=\ker $\tiny$\left[ \begin{smallmatrix} 0 & -A  & 1+AB & |&1   & \|  &0\\ 0  & 0 & -A  & | & 0  & \| &0\\ 0 & 0 & 0 & |  & 0 & \| &0 \\ - & - & - & - &-   & - &- \\ 0 & 0 & 1-A & | & 0  & \| &0 \\  = & = & = & = &=   & = &=\\ 0 & 0 & 0 & |& 0 & \| &0  \end{smallmatrix} \right]$.\normalsize\ \vspace*{-30pt}\end{center}    \hfill $\square$
\end{example}
\vspace*{15pt}

\noindent {\bf $\S$ 5.} \textsc{Concluding Remarks}\\[1ex]

\noindent This paper, as in the paper of Halmos~\cite{Halmos}, is set in \Cn\ for some positive integer $n$.    Our proofs depend on the Jordan Canonical Form and our results depend on the factorization of the characteristic polynomial of $A$ into linear factors, so these results are valid for real matrices whose characteristic polynomial has only real roots, but does not address similar questions about matrices whose characteristic polynomial has complex roots that are not real numbers.   

The construction in Theorem A4 is not the only method that works, nor is the matrix  $M$ with $\M= \ker M$ unique.  Other methods may be advantageous to use, for example, in computational programming settings or in tandem with basic linear algebra concepts such as solving a system of equations.  This last section provides an algorithm that lends itself to this scenario.  The different cases proceed similar to the above analysis, and again the heart of the algorithmic process (basically describing the full procedure) is for any irreducible invariant subspace (having a chain structure) for a nilpotent Jordan canonical form matrix $J$.  Given such a  subspace $\M= \{ \vec{v}, J\vec{v}, J^2\vec{v}, \ldots , J^{k-1}\vec{v} \}$, construct $M$ using the following algorithm: 

\begin{enumerate}[label=Step \arabic*:,itemsep=-1mm]
\item Form the $k \times n$ matrix $X$ whose rows  are the elements of \M\ ({\it thought of as row vectors} and written in terms of the given elementary basis) in reverse order of the list above. Note $\M = \mbox{RowSpace}(X)$, and so $\M^\perp = \mbox{ker }X$.
\item Row reduce $X$ and use each row to form the system of linear equations corresponding to the kernel elements.
\item Use these equations to form $M$, blocked and consistent with the format in Lemma 1,  so that $\mbox{RowSpace}(M)= \mbox{ker }X=\M^\perp$  and hence $\mbox{ker }M=\M$, in the following way: 
\begin{enumerate}[label=Step \Alph*:,itemsep=-1mm]
  \item Start with the first equation, which has the form $\sum c_{ij}x_{ij}=0$. If there are an even number of nonzero coefficients $c_{ij}$, then set the first half of the $x_{ij}$ terms equal to $1/c_{ij}$ and the last half equal to $-1/c_{ij}$.  (These terms collectively satisfy the equation.)  If there are an odd number, say $p>1$, of coefficients, set the first two of the $x_{ij}$ terms equal to $1/(2c_{ij})$, the next $(p-3)/2$ equal to $1/c_{ij}$ and the last $(p-1)/2$ equal to $-1/c_{ij}$.  (These terms collectively satisfy the equation.) Of course, if $p=1$, then $x_{ij}=0$.  For any case, enter each of the $x_{ij}$ values into the $\vec{e}_{ij}$ column position of the first row of $M$. 
  \item Since $M$ is block upper-triangular Toeplitz of the form in Lemma 1, many of its entries are automatically zero, and each item entered in Step A extends diagonally down its block.  Use these facts to fill in additional entries of $M$.
  \item Use each successive linear equation, substituting the values for variables already obtained, to determine additional entries in $M$ that correspond to any other variable $x_{ij}$ having a nonzero coefficient in at least one of the linear equations, using the same techniques as in Steps A and B.  (If substituted variable values are nonzero, then simple adjustments might need to be made for the $x_{ij}$ choices.  For example, if the equation becomes $\sum c_{ij}x_{ij}=C$, where $C$ is a nonzero constant, then set the first variable $x_{ij}$ equal to $C/c_{ij}$ and then apply the formulation from Step A to the remaining elements in the equation.)  
  \item Some variables $x_{ij}$ may not have any nonzero coefficient in any of the linear equations; they are nondeterministically   free.  Fill in their corresponding entries, in various rows of $M$, with 0's or 1's, in a way that correctly forms the rank of $M$.  Make sure $\mbox{Rank }M = n - \mbox{dim ker } M $. 
\end{enumerate}
\end{enumerate}

\begin{example}{F.} $J \sim (3-1)$. Suppose the Jordan blocks of $J$ are of size $3\times 3$ and $1 \times 1$. Let $\M=\mbox{ span } \{ \vec{e}_{11}, \vec{e}_{12}+\alpha \vec{e}_{11} + \beta \vec{e}_{21} \}$.  Then:

\begin{enumerate}[label=Step \arabic*:,itemsep=-1mm]
\item $X=\left[ \begin{smallmatrix} 1 & 0 & 0 & 0 \\ \alpha  &  1 & 0 & \beta   \end{smallmatrix} \right]$.

\item $\mbox{RowReduce}[X]= \left[ \begin{smallmatrix} 1 & 0 & 0 & 0 \\ 0 &  1 & 0 & \beta   \end{smallmatrix} \right]$. Hence \vspace{-8pt}
\[ \begin{array}{cccc}
1 x_{11}& &= & 0 \\
 &\ 1 x_{12}\ +\ \beta x_{21} & = & 0
\end{array}\]
\vspace*{-18pt}
\item Form $M$:
\begin{enumerate}[label=Step \Alph*:,itemsep=-1mm]
  \item Start with the first equation, which says $x_{11}=0$.  
  \item Using this value with additional facts in Step B, $M=$\tiny$\left[ \begin{smallmatrix} 0& \square & \square & | & \square \\ 0&0 & \square &| &0\\ 0 & 0 & 0 &  | &0  \\ -  &  - & - & - & -  \\ 0  &  0 & \square & | & \square   \end{smallmatrix} \right]$\normalsize, where the element in each square is not yet determined.
  \item The second equation $1  x_{12}\ +\ \beta  x_{21}  =  0$ determines $x_{12}=1$, and $x_{21}=-1/\beta$.  Using these values with additional facts in Step B, $M=$\tiny$\left[ \begin{smallmatrix} 0& 1 & \square & | & -1/\beta \\ 0&0 & 1 & |&0\\ 0 & 0 & 0 & |&  0  \\ -  &  - & - & - & -  \\ 0  &  0 & \square & |& \square   \end{smallmatrix} \right]$\normalsize.
  \item To obtain $\mbox{Rank }M = n - \mbox{dim ker } M =4-2=2$, the last two rows of $M$ are zero. Setting the last unknown entry in the first row equal to 0 (as $x_{13}$ is nondeterministically free),  \begin{center} $M=$\tiny$\left[ \begin{smallmatrix} 0& 1 & 0 & | &-1/\beta \\ 0&0 & 1 &| &0\\ 0 & 0 & 0 &  | &0  \\ -  &  - & - & -  & -  \\ 0  &  0 & 0 & | &0   \end{smallmatrix} \right]$\normalsize.\end{center} 
\end{enumerate}
\end{enumerate} \vspace*{-15pt} \hfill $\square$ \vspace*{18pt}
\end{example}

In~\cite{Halmos}, not only did Halmos prove that every invariant subspace is a ``commuting kernel,'' he also proved that every invariant subspace $\N$ for a given $n \times n$ matrix $A$ is a ``commuting range.'' That is, there is an $n \times n$ matrix $R$ with $R \in \{A\}^\prime$ such that $\N=Range(R)$. A quick application of Theorem A to $\N^{\perp}$, an invariant subspace of the adjoint $A^{*}$, yields the following parallel result.

\begin{cor} {\it For a given $n \times n$ matrix $A$ over $\mathbb{C}$ and an $A$-invariant subspace \N, there exists an $n \times n$ matrix $R=QP(T^{*})^{-1}Z^{*}P^tQ^{-1}$ over $\mathbb{C}$, where  $P^t$ denotes the transpose of $P$, $Q$ puts $A$ into Jordan form $J$, and: 
\begin{enumerate}  [topsep=1pt, label=\alph*) \ ]
\item $\N=Range\, R$;
\item $R \in \{A\}^\prime$;
\item $T$ and $P$ provide a standard change of variables and row selection matrix, respectively;
\item $Z$ is a diagonal or near-diagonal block matrix whose non-zero blocks are each a power of a finite-dimensional backward shift.\ 
\end{enumerate}
Furthermore, the subspace \N \ is marked if and only if the construction produces $Z$ and $T$ that are  $n \times n$ with $Z$ block-diagonal and $P$ the identity. } 
\end{cor}

\begin{proof} Apply Theorem A to $\N^{\perp}$, an invariant subspace for  $A^{*}$, to produce 
$M=PZT^{-1}P^t \in \{J^{*}\}^\prime$ with $\M^{\perp}=\ker M$, which  translates to
$R=QM^{*}Q^{-1}=Q(PZT^{-1}P^t)^*Q^{-1}=QP(T^{*})^{-1}Z^{*}P^tQ^{-1}$, with  $R\in \{A\}^\prime$ and $\N=Range\, R$.
\end{proof}

Finally, we make note of a similar investigation, one that is parallel to and concerns a generalization of the discussion in this paper.  To that end, let $N$ be a linear transformation on \Cn\ and let \N\ be a hyperinvariant subspace of dimension $k<n$ for $N$.   As previously mentioned, for any complex number $\alpha$,  the transformations $N$ and $N+\alpha I$ have the same invariant subspaces.  Moreover, $N$ commutes with a transformation $A$ if and only if $A$ commutes with $N+\alpha I$, so $N$ and $N+\alpha I$ have the same hyperinvariant subspaces.\  The investigation starts with the following observation, a parallel extension of the fact -- the very first one mentioned in this paper -- that $\ker N$ is $A$-invariant when $N$ is in the commutant of $A$:

\begin{thm}\label{translates}  Let $N$ and $A$ be linear transformations on \Cn.  If $N$ commutes with  $A$, then for each positive integer $k$ and each $\alpha$ in \Co,  the nullspace of $(N-\alpha I)^k$ is an $A$-invariant subspace.
\end{thm}
\begin{proof}
Let $A$ be a linear transformation that commutes with the transformation $N$, let 
$\alpha$ be a complex number, let $k$ be a positive integer and let $\vec{v} \in \ker (N-\alpha I)^k$.   Since $A$ commutes with $N$ implies $A$ commutes with $(N-\alpha I)^k$, 
\[ (N-\alpha I)^k (A\vec{v}) = A (N-\alpha I)^k \vec{v}=A\vec{0}=\vec{0} \]
and $A\vec{v} \in \ker (N-\alpha I)^k$.   Since this is true for each $\vec{v}$ in $\ker (N-\alpha I)^k$, this  means  the nullspace of $(N-\alpha I)^k$ is an invariant subspace for $A$.
\end{proof}

On the other hand, the following example shows that the converse is {\bf not} true.

\begin{example}{G.} Let $N$ be the $3 - 2$ nilpotent transformation on $\Co^5$ that has matrix with respect to the usual elementary basis: $ N\sim $\tiny$\left[\begin{smallmatrix} 0 &1& 0 &0 &0\\ 0 &0& 1 &0 &0\\0 &0& 0 &0 &0\\0 &0& 0 &0 &1\\ 0 &0& 0 &0 &0 \end{smallmatrix}\right]$\normalsize. In the same basis, let $A$ be the transformation with matrix $A\sim $\tiny$\left[\begin{smallmatrix} 1&3& 5 &0 &0\\ 0 &1& 7 &0 &0\\0 &0& 1 &0 &0\\0 &0& 0 &2 &11\\ 0 &0& 0 &0 &2 \end{smallmatrix}\right] $\normalsize.
It is easy to check that the nullspace of $N$ is spanned by $\{ \vec{e_1}, \vec{e_4}\}$, the nullspace of $N^2$ is spanned by $\{\vec{e_1}, \vec{e_2}, \vec{e_4}, \vec{e_5}\}$, the nullspace of $N^k$ is $\Co^5$ for all $k\geq 3$, and that all of these subspaces are invariant subspaces for $A$.  On the other hand, $(NA-AN)\vec{e_3} = 4\vec{e_1}\neq 0$; hence $A$ and $N$ do not commute! \hfill $\square$ 
\end{example}  

\bigskip
In the last example, the invariant subspace lattice for $A$ is contained in the invariant subspace lattice for $N$. But then $A$ is not reflexive in the sense of \cite{Deddens}, since the failure of $N$ to commute with $A$ cannot allow $N$ to be a polynomial in $A$. Also, a reasonable interpretation of this last result is that Halmos' 1971 theorem is, in this investigative vein, the most general result possible, and (in this sense) this paper's construction fully characterizes the relationship between a finite linear transformation $A$ on $\mathbb{C}^n$, its invariant subspaces, and the elements in the commutant $\{ A \}^\prime$.




  
\end{document}